\documentclass[12pt]{amsart}
\pdfoutput=1
\usepackage{amsthm}
\usepackage{epsfig}
\usepackage{color}

\usepackage{float}

\usepackage{amsfonts, amssymb, curves, eucal,pxfonts, txfonts}

\let\oldmarginpar\marginpar
\renewcommand\marginpar[1]{\-\oldmarginpar[\raggedleft\footnotesize #1]%
{\raggedright\footnotesize #1}}

\usepackage[backref=page]{hyperref}

\newtheoremstyle{theorem}{}{}{\itshape}{\parindent}{\scshape}{.}{ }{}
\theoremstyle{theorem}
\newtheorem{theorem}{Theorem}[]
\newtheorem{prop}{Proposition}[section]

\newtheorem{lemma}[prop]{Lemma}
\newtheoremstyle{remark}{}{}{}{}{\bf}{.}{ }{}
\theoremstyle{remark}
\newtheorem*{remark}{Remark}
\newtheorem*{definition}{Definition}
\newtheorem{example}{Example}[section]

\newcommand{\nats}{\mathbb{N}}
\newcommand{\ra}{\rightarrow}

\DeclareSymbolFont{rsfs}{U}{rsfs}{m}{n}
\DeclareMathSymbol{\ACC}{A}{rsfs}{"41}

\begin{document}

\title[Expansive surface attractors]%
{Classification of expansive attractors on surfaces}

\author{Marcy Barge}
\address{Department of Mathematics, Montana State University, Bozeman, MT 59717}
\email{barge@math.montana.edu}

\author{Brian F. Martensen}
\address{Department of Mathematics and Statistics, Minnesota State University, Wissink 273, Mankato, MN 56001}
\email{martensen@mnsu.edu}

\begin{abstract}
We prove the conjecture of F. Rodriguez Hertz and J. Rodriguez Hertz (\cite{herher}) that every nontrivial transitive expansive attractor of a homeomorphism of a compact surface is a derived from pseudo-Anosov attractor.
\end{abstract}

\maketitle

\section{Introduction}
The prototype for the sort of object we consider in this paper is the derived from Anosov attractor created by Smale \cite{sma}. In this example, the unstable foliation of a hyperbolic toral automorphism is split open along a single leaf to create a hyperbolic 1-dimensional expanding attractor in the 2-torus. Hyperbolic 1-dimensional expanding attractors can be created similarly on surfaces
of higher genus by splitting open the unstable foliation of a pseudo-Anosov homeomorphism along all leaves associated with singularities. These attractors have a very regular structure: locally they are the product of a Cantor set with an arc, globally they are the inverse limit of an expanding map of a branched 1-manifold. This description, and a complete classification of hyperbolic 1-dimensional expanding attractors, in any ambient dimension, is given by Williams in \cite{wil2}.

Remarkably, if hyperbolicity is dispensed with altogether, and expanding is weakened to expansive,
one can still say a lot about the structure of attractors on surfaces. In the article \cite{herher} Rodriguez Hertz and Rodriguez Hertz prove that every expansive surface attractor has a local product structure at all but finitely many points. They conjecture that if such an attractor is also transitive, then it must be a derived from
pseudo-Anosov attractor. That is, it must be conjugate with an attractor obtained by splitting open the unstable foliation of a pseudo-Anosov (or Anosov) homeomorphism along finitely many leaves. We
prove this conjecture here.

It follows from our result that transitive and expansive surface attractors are very nearly hyperbolic.
Indeed, simply unzipping finitely many unstable branches and splitting finitely many periodic orbits turns
such an attractor into a hyperbolic attractor. Hyperbolic one-dimensional attractors on surfaces are either conjugate with one-dimensional substitution tiling spaces, if orientable, or  are double-covered by one-dimensional substitution tiling spaces if not orientable (\cite{holmar}).
A consequence is that for transitive and expansive attractors on surfaces, topology more-or-less determines dynamics: if  two such attractors $A$ and $B$, for homeomorphisms $f$ and $g$, are homeomorphic, then there are positive integers $m$ and $n$ so that $f^m|_A$ and $g^n|_B$ are conjugate (see \cite{barswa} - this also follows from \cite{mosher}).

\section{Definitions, statement of the main theorem, and an example}

\begin{definition}
A homeomorphism $f$ of a  compact surface is {\em pseudo-Anosov} provided there exists a {\em dilatation} $\lambda >1$ and a pair of invariant, continuous, transverse (except at singularities) foliations $\mathcal{F}^{s}$ and $\mathcal{F}^{u}$. The foliations carry transverse measures $\mu^{s}$ and $\mu^{u}$ which are expanded by precisely $\lambda$ under each iteration of $f$ and $f^{-1}$ , respectively, and possess a finite 
number of singularities near each of which  $\mathcal{F}^{s}$ and $\mathcal{F}^{u}$ are homeomorphic with the foliations of $\mathbb{C}$ by curves of constant real and imaginary parts (modulo sign if $k$ is odd) of $z^{k/2}$ for some $k\in\nats\setminus\{ 2\}$. 
\end{definition}

\begin{remark}The above definition includes {\em Anosov homeomorphisms} (no singularities) as well as {\em relative pseudo-Anosov homeomorphisms} (those with 1-pronged, i.e., $k=1$, singularities).
\end{remark}

Speaking roughly, a derived from pseudo-Anosov attractor is an attractor obtained by `unzipping' finitely
many branches of the unstable foliation at finitely many periodic points of a pseudo-Anosov homeomorphism. As it is rather cumbersome to make this precise (see the appendix for a careful description of unzipping), we opt for the following definition.
\begin{definition}
An attractor $A$ for a homeomorphism $f:M\to M$ is {\em derived from pseudo-Anosov} provided there is a pseudo-Anosov surface homeomorphism $g:S\to S$ and a continuous surjection $\pi:A\to S$ with the properties:

\begin{enumerate}
\item $g\circ\pi=\pi \circ f$; 

\item there are finitely many periodic points $p_1,\ldots,p_n$, $n\ge 1$, of $g$ and branches $B_{i,j}$ of the unstable foliation of $S$ at $p_i$, $j=1,\ldots,m(i)\ge1, $ for each $i$, so that $\pi$ is one-to-one off
$\pi^{-1}(\cup B_{i,j})$;

\item $\pi$ is exactly two-to-one on $\pi^{-1}(\cup (B_{i,j}\setminus \{p_i\}))$; and

\item if $\Phi_{i,j}:\mathbb{R}^+\cup\{0\}\to B_{i,j}$ parameterizes $B_{i,j}$, then $diam(\pi^{-1}(\Phi_{i,j}(t)))\to 0$ as
$t\to \infty$.
\end{enumerate}
\end{definition}
\begin{definition}
An invariant set $A$ for $f:M\to M$ is an {\em expansive attractor} provided there is an $\alpha>0$ so
that $\cap_{n\ge 0}f^n(B_{\alpha}(A))=A$ and $sup_{n\in \mathbb{Z}}d(f^n(x),f^n(y))>\alpha$ for all
$x\ne y\in A$. (Here $d$ is a metric on $M$ and $B_{\alpha}(A)$ is the $\alpha$-neighborhood of $A$.)
\end{definition}
The main result of this paper is the following classification theorem.
\begin{theorem}\label{classificationtheorem} If $A$ is an expansive and transitive attractor for a homeomorphism of a compact surface, and $A$ is not a single periodic orbit, then $A$ is derived from pseudo-Anosov.
\end{theorem}
We end this section with an example of a well-known hyperbolic attractor which provides a  demonstration of how 1-pronged singularities will be handled in the proof of the main result.

\begin{example}\label{plykin-like}

Recall the construction of the Plykin attractor:

Let $A : \mathbb{T}^{2}\ra\mathbb{T}^{2}$ be the hyperbolic torus automorphism 
induced by \[A=\left( \begin{array}{cc}2 & 1 \\ 1 & 1 \end{array} \right). \]
The quotient  space $\mathbb{T}^{2} /\sim$, with $x\sim -x$, is homeomorphic to the 2-sphere 
$S^{2}$ and the quotient map $\pi : \mathbb{T}^{2}\ra S^{2}$ is a branched double cover with $4$ branch points. The map $A$ 
induces a homeomorphism $f_{A} : S^{2}\ra S^{2}$ which is pseudo-Anosov with 1-pronged singularities at the branch 
points.

The map $f_A$ is not expansive near the $1$-prong singularities. (It is, rather, {\em continuum-wise expansive} - every nontrivial continuum achieves diameter bigger than $\alpha$ under iteration). However,  unzipping the unstable manifold of each $1$-prong singularity (see Figure~\ref{split_fig}) produces the  expansive Plykin attractor.
\end{example} 

\begin{figure}[H]
\includegraphics[width=3in]{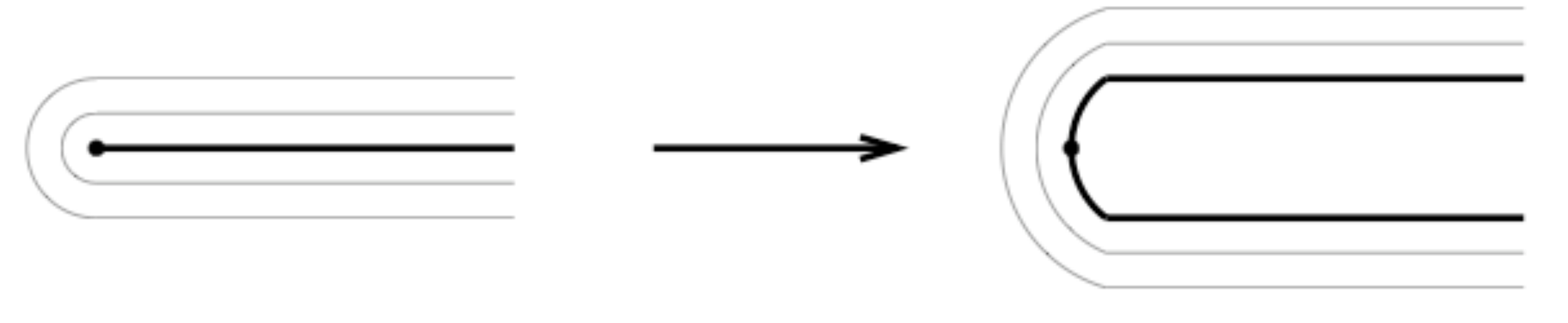}
\caption{The unstable manifold of a $1$-prong singularity is unzipped. All points accessible from the new complementary domain lie on the unstable manifold of a hyperbolic  periodic point.}\label{split_fig}
\end{figure}

The method of proof of the classification theorem will be to reverse this construction. That is, beginning with an expansive attractor, we will zip up pairs of accessible branches to produce a homeomorphism of a surface.  As in the Plykin example, the resulting homeomorphism will be nonexpansive near any $1$-prong singularities. If there are such singularities, the homeomorphism can be lifted, via a branched double cover, to a homeomorphism which is expansive on an entire surface. From a theorem proved independently by Hiraide and Lewowicz, we will know that this final homeomorphism is pseudo-Anosov, and we will deduce from this that the original attractor is derived from pseudo-Anosov.

\section{Local Product Structure}

We summarize here the results we will need from \cite{herher}. We assume that $A\subset M$ is a transitive
and expansive 
attractor consisting of more than a single periodic orbit, with expansiveness constant $\alpha$, for the surface homeomorphism $f:M\to M$. We also assume that $M$ is connected and that $A\ne M$.\\
\\
 For each $n\in \mathbb{N}$, $n\ge 2$, and
$k\in \{0,\ldots,n-1\}$, let $\psi_{k,n}$ be the homeomorphism from the sector $\{(r,\theta):r\ge 0, 2k\pi /n \le \theta \le 2(k+1)\pi /n\}$ to the upper half-plane given by $\psi_{k,n}(r,\theta)=(r,n\theta /2-k\pi)$.\\
\\
\begin{theorem}\label{productstructure}(\cite{herher}) For each $x\in A$ there is a homeomorphism $\phi$ from the unit ball $\mathbb{B}$ centered at the origin in $\mathbb{R}^2$ onto a neighborhood of $x$ in $M$ with $\Phi((0,0))=x$ and with the property that, for some 
$n=n(x)\in \mathbb{N}$, $n\ge2$: $\psi_{k,n}(\phi^{-1}(A))= ( \mathbb{R}\times C_k)\cap \mathbb{B}$, the right hand side in rectangular coordinates, where
$C_k$ is one of two types. Either:\\
(i) $C_k$ is a compact 0-dimensional subset of $[0,1/2]$ with no isolated points (that is, a Cantor set) and with $0\in C_k$, or\\
(ii) $C_k=\{0\}$.\\
Furthermore, if type (ii) occurs for some $k$, then type (i) occurs for $k\pm1$ (mod($n$)).
\end{theorem}

\begin{figure}[H]
\includegraphics[width=2in]{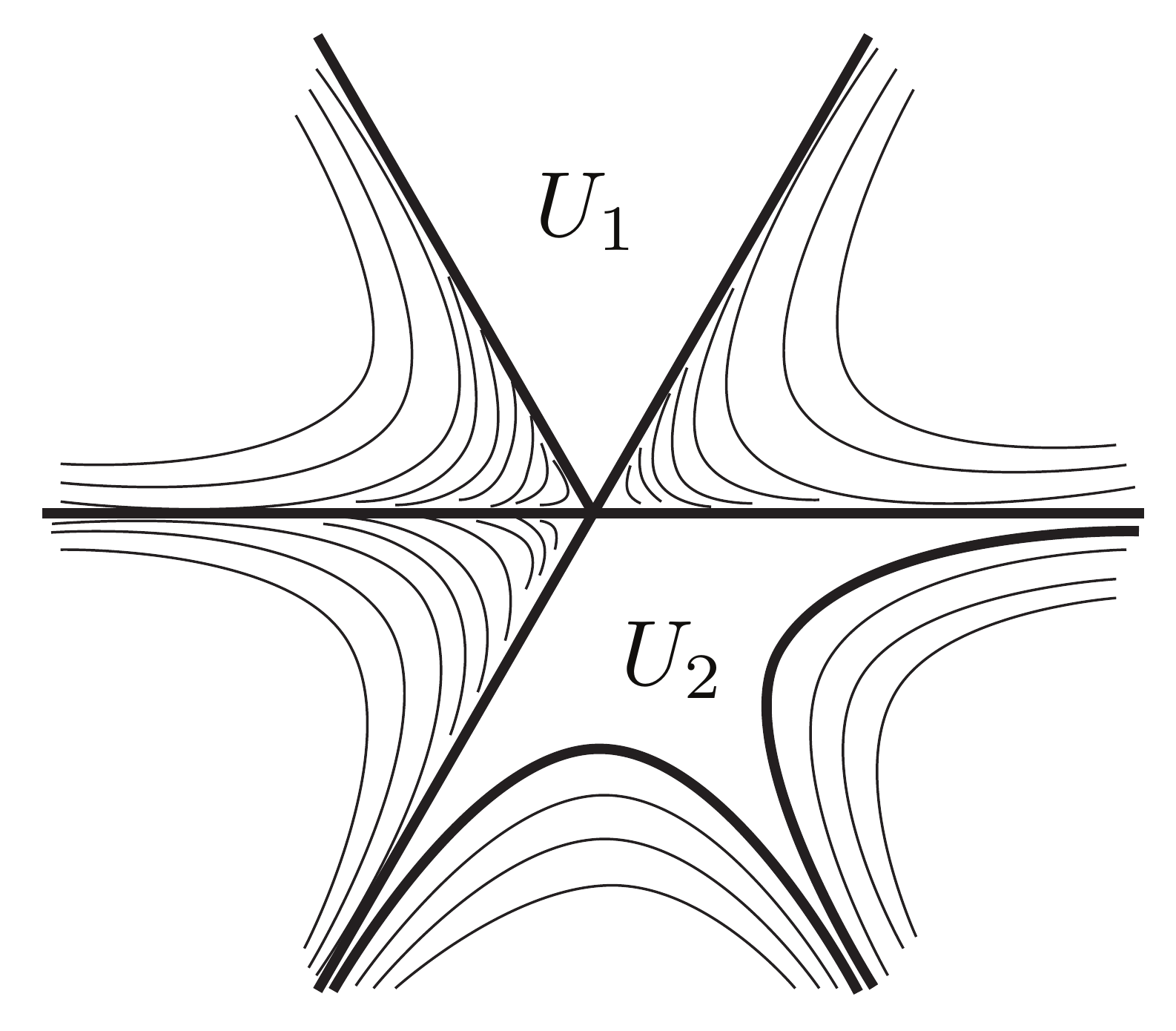}
\caption{A $5$-prong singularity of an expansive attractor.  The singularity is accessible from complementary domains $U_1$ and $U_2$.}\label{dpA_fig}
\end{figure}

If the $n=n(x)$ in the proposition is 2, then $x$ is a {\em regular point}. Otherwise, $x$ is called a {\em singular point} or an {\em n-pronged singularity}. (The possibility $n(x)=1$, i.e., $x$ is an {\em epin\'{e}}, is disallowed by the transitivity assumption.) We will call the neighborhood $\phi(\mathbb{B})$ of
$x$ a {\em product}, or {\em singular product, neighborhood} of $x$.
It is clear that there are only finitely many singular points in $A$.
\begin{remark} In \cite{herher}, the product structure obtained is relative to $A$. That is, a regular point in $A$ has a neighborhood, in $A$, homeomorphic with the product of a compact 0-dimensional set with an arc, etc. This is easily extended to ambient neighborhoods, as in Theorem~\ref{productstructure}, using \cite{tymwal}.
\end{remark}
For $x\in A$  and $\epsilon \ge 0$, let $W_{\epsilon}^u(x):=\{y\in M:d(f^{-n}(y),f^{-n}(x))\le \epsilon$ for all $n\ge 0\}$ and let $CW_{\epsilon}^u(x)$ denote the connected component of $W_{\epsilon}^u(x)$ containing $x$. Similarly, define $W_{\epsilon}^s(x)$ and $CW_{\epsilon}^s(x)$, replacing $-n$ by $n$.
\begin{lemma}\label{prongs}(\cite{herher}) For $\epsilon<\alpha/2$, $CW_{\epsilon}^u(x)$ is an $n-od$ contained in $A$, $n=n(x)$ as in Theorem~\ref{productstructure} above. Also, there is a uniform $\delta>0$ so that each prong of  $CW_{\epsilon}^u(x)$ meets the boundary of the $\delta$-ball centered at $x$. Furthermore, if $\beta$ and $\gamma$ are any two prongs of  $CW_{\epsilon}^u(x)$ and $\beta_0$ and $\gamma_0$ are the components of $\beta\cap \bar{B}_{\delta }(x)$ and $\gamma\cap \bar{B}_{\delta }(x)$ containing $x$, then the component of $CW_{\epsilon}^s(x)\cap \bar{B}_{\delta}(x)$ containing $x$ separates $\beta_0\setminus \{x\}$ from
$\gamma_0\setminus\{x\}$ in $\bar{B}_{\delta}(x)$.
\end{lemma}
\section{Complementary Domains}\label{compdomains}

\begin{prop}\label{finite_complements}
Let $A\subset M$ be an expansive attractor for the homeomorphism $f:M\to M$ of the compact surface $M$.  Then, $A$ has finitely many complementary domain in $M$.
\end{prop}

\begin{proof}
Suppose that $M\setminus A$ has infinitely many complementary domains. Let $\alpha$ be as in the definition of expansive attractor. There is then a complementary domain $D$ with the property that, for all but finitely many $n$, all the points of $f^n(D)$ are within $\alpha$ of $A$. Then there is $x\in D$ that remains within $\alpha$ of $A$ under all iterates of $f$. But then $x\in A$, since $A=\cap_{n=0}^\infty f^n(B_\alpha(A))$.

\end{proof}

We will use the theory of prime ends on several occasions in what follows. Here we set terminology and recall a few basic facts.\\
Suppose that $f:M\to M$ is a homeomorphism of a closed surface $M$ and that $U\subset M$ is a
connected open set, invariant under $f$ with the properties: $M\setminus U$ is nonempty, has finitely many components, and none of these components is a singleton. If $c$ is a closed arc in $M$ with distinct  endpoints $x$ and $y$ such that $c\cap U=c\setminus \{x,y\}$, and $diam(c)$ is sufficiently small, then $c$ separates $U$ into two connected nonempty pieces. Such an arc is called a {\em cross-cut} of $U$. A sequence $\{c_k\}_{k=1}^{\infty}$ of disjoint cross-cuts is called a {\em chain} of
cross-cuts if $diam(c_k)\to 0$ as $k\to \infty$ and $\mathring{c}_{k+1}$ separates $\mathring{c}_k $ from $\mathring{c}_{k+2}$ in $U$ for each $k\in \mathbb{N}$ (here $\mathring{c}$ denotes $c$ less its endpoints).
The chain of cross-cuts $\{c_k'\}$ {\em divides} the chain of cross-cuts $\{c_k\}$ provided for each $k$ there is an $L$ so that $\mathring{c}_{k+1}$ separates $\mathring{c}_l$ from $\mathring{c}_k$ in $U$ for all $l\ge L$. Two chains of cross-cuts are {\em equivalent} if each divides the other. (Since we are requiring that diameters of cross-cuts in a chain go to zero, chains are equivalent if either one divides the other.) A {\em prime end } of $U$ is an equivalence class of cross-cuts and we denote the collection of all prime ends of $U$ by $\partial\hat{U}$. We let
$\hat{U}:=U\cup\partial\hat{U}$ and refer to the points $x\in U\subset \hat{U}$ as {\em prime points}. If $\{c_k\}$ is a chain of cross-cuts representing a prime end $\mathbf{p}$ of $U$ and $W_k$ is the component of $U\setminus c_k$ that contains $\mathring{c}_{k+1}$, let $\hat{W}_k:=W_k\cup \{\mathbf{q}:\mathbf{q}$ is a prime end of $U$ with representing chain of cross-cuts $\{c_l'\}$ such that $\mathring{c}'_l\subset W_k$ for all $l\}$. We topologize $\hat{U}$ by declaring that the sets $\hat{W}_k$ form a neighborhood base at $\mathbf{p}$ and that the sets $V$, $V$ open in $U$, $x\in V$, form a neighborhood base at the prime point $x$. \\
\\
With this topology, $\hat{U}$ is a compact surface with boundary $\partial \hat{U}$ and the {\em induced
map} $\hat{f}:\hat{U}\to \hat{U}$ defined by $\hat{f}(x)=f(x)$ on prime points and $\hat{f}([\{c_k\}])=
[\{f(c_k)\}]$ on prime ends, is a homeomorphism. (See, for example, [Mather].)\\
\\
The {\em principal set} of a prime end $\mathbf{p}$ is the set of all points $x$ with the property that $\{x\}=\lim_{k\to \infty}c_k$ for some chain of cross-cuts $\{c_k\}$ representing $\mathbf{p}$. The principal set is a subcontinuum of the boundary of $U$. If the principal set of $\mathbf{p}$ is a singleton, say $\{p\}$, then $p$ is accessible from $U$ and $\mathbf{p}$ is called an {\em accessible prime end}. If $c$ is a nontrivial continuum in $M$ of small diameter with $c\setminus U=\{x,y\}$, there is a {\em corresponding continuum} $\mathbf{c}$ in $\hat{U}$ with $\mathbf{c}\cap \partial \hat{U}=\{\mathbf{x},\mathbf{y}\}$, $\mathbf{x}$ and $\mathbf{y}$ accessible prime ends with principal sets $\{x\}$ and $\{y\}$, resp., and
$\mathbf{c}\setminus \{\mathbf{x},\mathbf{y}\}=c\setminus \{x,y\}$. Moreover, $\mathbf{c}$ has small diameter (that is, as $diam(c)\to 0$, $diam(\mathbf{c})\to 0$ in any compatible metric on $\hat{U}$).
\begin{prop}\label{ratrot} Let $A\subset M$ be a transitive expansive attractor of a homeomorphism $f:M\to M$ and let $U$ be a complementary domain of $A$. There are then finitely many periodic points $p_1,\ldots,p_n$
in $A$ accessible from $U$, two of the branches of $W^u(p_i)$ consist entirely of points accessible 
from $U$ for each $i$, and every point of $A$ accessible from $U$ lies on one of these branches. 
\end{prop}
\begin{proof} Assume for simplicity that $f(U)=U$. Let $\hat U$ be the surface of prime points and prime ends
associated with $U$ and let $\hat f:\hat U\rightarrow \hat U$ be the homeomorphism induced by $f$. By passing to a power of $f$, we may assume that each component of $\partial \hat{U}$ is invariant under
$\hat{f}$ and that $\hat{f}$ is orientation preserving on each component of $\partial \hat{U}$. To further simplify notation we suppose that $\partial \hat{U}$ is connected.\\
 We
first prove that the map $\hat f\mid _{\partial \hat U}:\partial \hat U\rightarrow \partial \hat U$ on 
the circle of prime ends has rational rotation number.
\\
Let $\mathbf{p}$ be a prime end
in $\partial \hat U$ that is recurrent under $\hat f$: say $\hat f^{n_i}(\mathbf{p})\to 
\mathbf{p}$ as $i\to \infty$. Let $\{c_k\}$ be a sequence of cross-cuts of $U$ defining $\mathbf{p}$
with $diam(c_k)\to 0$. Let $x_k$ and $y_k$ be the endpoints of $c_k$. Then $x_k$ and $y_k$ are accessible
from $U$ and are nonsingular for large $k$. If  $x_k$ and $y_k$ are on the same unstable arc $\gamma \subset A$ (with endpoints  $x_k$ and $y_k$), then $\gamma$ consists entirely of points accessible from $U$ one of which, say $p$, determines $\mathbf{p}$. Then, for large $i$, $f^{n_i}(p)\in \gamma$. If  $f^{n_i}(\gamma)\subset \gamma$ or $f^{n_i}(\gamma)\supset \gamma$ then $f^{n_i}$ fixes a point in $\gamma$
that is accessible from $U$ and $\hat f$ has rational rotation number on $\partial \hat U$. Otherwise, let $\hat{\gamma}$ be the arc corresponding to $\gamma$ on $\partial \hat U$. Then either $\cup_{m=1}^{\infty}\hat{f}^{mn_i}(\hat{\gamma})=\partial \hat U$, in which case $\cup_{m=1}^{\infty}\hat{f}^{mn_i}(\gamma)$ is a topological circle in $A$ and invariant under $f^{n_i}$, which is ruled out by expansivity, or $\cup_{m=1}^{\infty}\hat{f}^{mn_i}(\hat{\gamma})$ is a half-open arc in $\partial \hat U$ whose open endpoint is fixed by 
$\hat{f}^{n_i}$. In the latter case also $\hat f$ has rational rotation number on $\partial \hat U$. 
\\
We may assume then that $x_k$ and $y_k$ are not on the same unstable arc. By considering the local product structure in a neighborhood of an accumulation point of $\{c_k\}$ and
using Lemma ~\ref{prongs}, we see that for large $k$ there is a continuum $\bar c_k \subset W^s_{\alpha/2}(x_k)$ with $\bar c_k \cap A=\{x_k,z_k\}$, $\bar c_k \backslash \{x_k,z_k\} \subset U$, and $z_k \in W^u_{\alpha/2}(y_k)$. The corresponding continuum $\bar{ \mathbf{c}}_k$ 
in $\hat U$ intersects $\partial \hat U$ in
accessible prime ends $\mathbf{x}_k$ and $\mathbf{z}_k$ which are endpoints of a nested sequence
of arcs $\mathbf{\beta}_k \subset \partial \hat U$ with $\cap \mathbf{\beta}_k= \{\mathbf{p}\}$. Since
$diam(f^{n_i}(c_k))\to 0$ as $i\to \infty$, $diam(\hat {f}^{n_i}(\bar {\mathbf{c}}_k))\to 0$ as $i\to \infty$ as well so
that $\hat {f}^{n_i}(\bar {\mathbf{c}}_k)\to \{\mathbf{p}\}$. But then $\hat {f}^{n_i}(\mathbf{\beta}_k)\subset \mathbf{\beta}_k$ for some $i$. Thus $\hat f$ has a periodic prime end and hence a rational rotation number.
\\
Passing to a power, we may assume that the rotation number of $\hat f$ on $\partial \hat U$ is $0$.
Let $\mathbf{p}$ be a prime end in $\partial \hat U$ fixed by $\hat f$ and let $\{c_k\}$ be a sequence
of cross-cuts defining $\mathbf{p}$. If the endpoints $x_k$ and $y_k$ of $c_k$ are on the same 
unstable arc $\alpha$ in $A$ then $\mathbf{p}$ is accessible from $U$, is repelling on $\partial \hat U$
under $\hat f$, and corresponds to an accessible point $p\in \alpha$ with $x_k$ and $y_k$ on distinct
branches of $W^u(p)$. If $x_k$ and $y_k$ are not on the same unstable arc, consider the stable continuum $\bar{\mathbf{c}}_k\subset \hat U$ constructed above. Since $\mathbf{p}$ is fixed by $\hat f$
and $diam(\hat{f}^n(\bar{\mathbf{c}}_k)\to 0$ as $n\to \infty$, $\mathbf{p}$ is attracting on $\partial \hat U$.
\\
Since every prime end fixed by $\hat f$ is either attracting or repelling, there are only finitely many of them and every non-attracting fixed prime end is on the unstable manifold (with respect to $\hat f$ on 
$\partial \hat U$) of a fixed repelling prime end. We saw above that the local unstable manifold of the 
fixed repelling prime ends consists of accessible prime ends associated with accessible points on two branches of the unstable manifold of an accessible fixed point. It follows that the global unstable manifold  of a fixed repelling prime end also has this property. 
\end{proof}
Given $f:M\to M$ with expansive attractor $A$, we would like to embed $A$ in a surface $M'$ so that
all components of $M'\backslash A$ are topological disks, and so that $f|_A$ extends to a homeomorphism $f':M'\to M'$ having $A$ as attractor. This would be straightforward were it true
that $\cup_{x\in A}W^s_{\epsilon}(x)$ is a neighborhood of $A$ for every $\epsilon>0$. However, we are unable to verify that this is the case (see Question~1.3 in \cite{herher}) and we are forced into the more elaborate procedure below (Lemma~\ref{step0}) that requires modification of $f$ arbitrarily near $A$. This
is one step in a construction that will produce (Proposition~\ref{step3}) an equivariant upper semi-continuous decomposition
of an expansive attractor $A'$ into a surface $\tilde{M}$ with the induced $\tilde{f}:\tilde{M}\to \tilde{M}$  pseudo-Anosov. The attractor $A$ is a quotient of $A'$ whose only nontrivial fibers are over
the finitely many periodic points in $A$ that are accessible by two or more inequivalent directions from
$M\backslash A$.

\begin{lemma}\label{step0}
Let $A\subset M$ be a transitive expansive attractor of a homeomorphism $f:M\to M$. There exist a surface $M_1$, a homeomorphism $f_1:M_1\ra M_1$, and a transitive expansive attractor $A_1\subset M_1$ such that each component of $M_1\setminus A_1$ is a topological disk and $f|_A$ is conjugate to $f_1|_{A_1}$.
\end{lemma}
\begin{proof}
Suppose that $U$ is a complementary domain of $A$ in $M$ that is not an open topological disk
and choose $n>0$ so that $f^n(U)=U$ and $\hat{f}^n(\mathbf{p})=\mathbf{p}$ for every periodic prime end in $\partial \hat{U}$. Let $\Delta$ be a component of $\partial \hat{U}$, let $g=f^n$ and let $\mathbf{p}_0,\ldots,\mathbf{p}_{m-1}$ be the fixed (under $\hat{g}$) inaccessible prime ends in $\Delta$, listed in cyclic order (see Proposition ~\ref{ratrot}).
Then each $\mathbf{p}_i$ is an attracting (on $\Delta$) fixed point of $\hat{g}$. Since each $\mathbf{p}_i$ is inaccessible, its continuum of principal points, $\mathcal{P}(\mathbf{p}_i)\subset A$, is nontrivial. As $g$ is expansive on $A$, not all points of $\mathcal{P}(\mathbf{p}_i)$ are fixed by $g$. For each
$i\in \{0,\ldots,m-1\}$ let $q_i\in \mathcal{P}(\mathbf{p}_i)$ be a principal point of $\mathbf{p}_i$ that is not fixed by $g$ and let $\{c_{i,k}\}$ be a chain of cross-cuts defining $\mathbf{p}_i$ such that $c_{i,k}\to \{q_i\}$ in the Hausdorff topology, as $k\to \infty$. Since each $\mathbf{p}_i$ is attracting on $\Delta$, the endpoints of 
$c_{i,k}$, as accessible prime ends, move towards $\mathbf{p}_i$ under $\hat{g}$ (at least for large $k$). Since $g(q_i)\ne q_i$, $g(c_{i,k})\cap c_{i,k}=\emptyset$ for all large $k$. It follows that if
$V_{i,k}$ is the component of $U\backslash c_{i,k}$ containing $\mathring{c}_{i,k+1}$, then $g(V_{i,k})\subset V_{i,k}$ for all $i$ and $k\ge K$ , some $K$. Furthermore, since $M$ is compact, we may take $K$ large enough so that $V_{i,K}$ is a topological disk. For $i=0,\ldots,m-1$, let $U_{\Delta,i}:=\cup_{t\ge 0}g^{-t}(V_{i,K})$. Then $U_{\Delta,i}$ is an open topological disk with circle of prime ends $\partial \hat{U}_{\Delta,i}$. 
\\
The prime ends in $\Delta$ fixed by $\hat{g}$ alternate, around $\Delta$, between repelling prime ends $\hat{p}$ associated with accessible points $p$ and attracting inaccessible prime ends
$\mathbf{p}$. Let us index the accessible prime ends so that the fixed points of $\hat{g}$ occur 
in cyclic order as ...,$\hat{p}_i$, $\mathbf{p}_i$, $\hat{p}_{i+1}$, $\mathbf{p}_{i+1}$,... , with subscripts taken mod($m$). The points $p_i$ and $p_{i+1}$ are also accessible from $U_{\Delta,i}$ and so determine prime ends in $\partial \hat{U}_{\Delta,i}$ which we denote by $\hat{p}_i^+$ and $\hat{p}_{i+1}^-$, resp. Let $\Gamma_i$ be the closed arc in $\partial \hat{U}_{\Delta,i}$ with endpoints $\hat{p}_{i+1}^-$ and $\hat{p}_i^+$ that does not contain $\mathbf{p}_i$. We construct a circle $\Gamma$ from the union of the $\Gamma_i$ by identifying $\hat{p}_{i+1}^-$ in $\Gamma_i$ with $\hat{p}_{i+1}^+$ in
$\Gamma_{i+1}$, indices taken mod($m$). Finally, we glue $\Gamma$ to the boundary of a disk $D_{\Delta}$ so that $\partial D_{\Delta}=\Gamma$. Let $U_{\Delta}:=D_{\Delta}\cup(\cup_{i=0,\ldots,m-1}
(\hat{U}_{\Delta,i}\backslash \{\mathbf{p}_i\})$. $U_{\Delta}$ is, topologically, a closed disk with $m$ points removed from its boundary.
\\
Let $\tilde{U}$ be the disjoint union of the $U_{\Delta}$, $\Delta$ a component of $\partial \hat{U}$. 
Let $M'$ be the surface obtained by gluing $\tilde{U}$ onto $M\backslash U$ in the natural way.
The complement of $A$ in $M'$ has one fewer non-contractible components than did the complement
of $A$ in $M$.
\\
Now iterate the above process until every complementary domain of $A$ in the final surface, call it $M_1$, is a topological disk. Note that a map $f_1$ extending $f$ on $A$ is naturally defined everywhere on $M_1$ except on the disks $D_{\Delta}$. If $\mathbf{x}\in \partial D_{\Delta}$, then $\mathbf{x}$ is a
prime end associated with a domain $U_{\Delta,i}$. There is then a well-defined prime end $\hat{f}(\mathbf{x})$ associated with a domain $V_{\hat{f}(\Delta),j}=f(U_{\Delta,i})$ where $V=f(U)$ and
$\hat{f}(\mathbf{p}_i)=\mathbf{q}_j$, $\mathbf{q}_j$ an inaccessible prime end in $\hat{f}(\Delta)\subset
\partial \hat{V}$; let $f_1(\mathbf{x}):=\hat{f}(\mathbf{x})$. Since $A$ attracts all points in a neighborhood of itself,
points $y$ accessible from $U_{\Delta,i}$ near $p_i$ (or near $p_{i+1}$) whose associated prime
ends $\hat{y}$ are on $\Gamma_i$ must be in $W^s(p_i)$ (resp., $W^s(p_{i+1}$)). Then for such $y$,
$f_1^{kn}(\hat{y})\to \hat{y}^+_i$ (resp. $\hat{y}_{i+1}^-$) on $\Gamma_i$ as $k\to \infty$. Thus the
periodic points $\hat{y}_i^{\pm}$ are attracting on $\partial D_{\Delta}$ under $f_1^n$. We thus may extend to a homeomorphism $f_1:D_{\Delta}\rightarrow D_{f(D_{\Delta})}$, for all $\Delta$, so that 
the orbits of the $\hat{p}_{i+1}^-\equiv \hat{p}_{i+1}^+$ are locally attracting in $\cup_{\Delta}D_{\Delta}$.
Now $f_1:M_1\rightarrow M_1$ is a homeomorphism of the surface $M_1$ with expansive attractor
$A\subset M_1$ and $f_1|_A=f|_A$. Moreover, each component of $M_1\backslash A$ is a topological
disk.
\end{proof}

\section{The Upper Semi-continuous Decomposition of A}\label{usc}
To demonstrate that a transitive expansive attractor is derived from pseudo-Anosov we will describe
how to `zip up' adjacent branches of unstable manifolds of accessible periodic points. The following proposition gives the zipping recipe ($x$ is to be identified with $o(x)$) for non-periodic accessible points.

\begin{prop}\label{step1a}
Let $A\subset M$ be a transitive expansive attractor of a homeomorphism $f:M\to M$ with constant $\alpha$.  Let $x\in A$ be a non-periodic point accessible from the complementary domain $U$.  Then $x$ lies on an unstable branch of an accessible periodic point $p$ and there exits a unique point $o(x)$ lying on an unstable branch of an adjacent accessible periodic point for which there is a continuum $c(x)$ such that:
\begin{enumerate}
\item $\{x,o(x)\}\subset c(x)$;
\item $c(x)\backslash \{x,o(x)\}\subset U$; and
\item $f^n(c(x))\subset W^s_{\alpha/2}(f^n(x))$ for sufficiently large $n$.
\end{enumerate}
Furthermore, the assignment $x\mapsto o(x)$ is continuous.
\end{prop}
\begin{proof}
 To simplify, assume
that the complementary domain $U$ is invariant under $f$ and that the periodic points in $A$ that are accessible from $U$
(which exist, according to Proposition ~\ref{ratrot}) are fixed by $f$. Assume also that each unstable branch of each accessible fixed point is invariant under $f$. Choose an $x\in A$ accessible from $U$, $x$
not fixed by $f$. Then the orbit of $x$ is contained in an unstable branch $B$ of an accessible fixed point $p$
(Proposition~\ref{ratrot}). Let $n_i\to \infty$ and $q\in A$ be such that $f^{n_i}(x)\to q$. In a product (or singular product) neighborhood $N$ of $q$ in $A$, infinitely many of the leaves of $N$ are visited by
$f^{n_i}(x)$. It follows from Lemma ~\ref{prongs} that for some $i$, $W^s_{\alpha/2}(f^{n_i}(x))$ meets leaves of
$N$ on both sides of the leaf containing $f^{n_i}(x)$. There is then a continuum $c(f^{n_i}(x))\subset W^s_{\alpha/2}(f^{n_i}(x))$ that  meets $N$ in points $f^{n_i}(x)$ and $y$, with $y\ne f^{n_i}(x)$ and $y$ accessible from $U$,
and with $c(f^{n_i}(x))\backslash \{f^{n_i}(x),y\} \subset U$. (Specifically, take $c(f^{n_i}(x))$ to be the component containing 
$f^{n_i}(x)$ in the intersection of  $W^s_{\alpha/2}(f^{n_i}(x))$ with $cl(U\cap N^*)$, $N^*$ the ambient
neighborhood corresponding to $N$ - see Theorem ~\ref{productstructure}.) Let $c(x):=f^{-n_i}(c(f^{n_i}(x)))$ and
$o(x):=f^{-n_i}(y)$.
\\
Thus, for each $x\in B$ there is a point $o(x)\ne x$ accessible from $U$ for which there is a continuum
$c(x)$ with the properties: $\{x,o(x)\}\subset c(x)$; $c(x)\backslash \{x,o(x)\}\subset U$; and
$f^n(c(x))\subset W^s_{\alpha/2}(f^n(x))$ for some $n$. Suppose that $x'$ is another point
accessible from $U$ and $c'$ a continuum with:  $\{x,x'\}\subset c'$; $c'\backslash \{x,x'\}\subset U$; and
$f^m(c')\subset W^s_{\alpha/2}(f^m(x))$ for some $m$. Then $diam(f^k(c'\cup c(x)))\to 0$ as $k\to \infty$
and by considering a product (or singular product) neighborhood of an $\omega$-limit point of $x$,
we see that, for some large $k$, $f^k(x')$ and $f^k(o(x))$ are on the same leaf in the product neighborhood and as close as desired, so that $f^k(x')\in W^u_{\alpha}(o(x))$, and also $f^k(x'),o(x)\in
W^s_{\alpha/2}(f^k(x))$. Expansivity implies that $f^k(x'))=f^k(o(x))$ and $x'=o(x)$. That is, the point
$o(x)$ is well-defined, independent of the details of the construction. Thus $o(o(x))=x$ for all accessible, non-fixed $x$. 
\\
It is not hard to see that $o$ is continuous: consider $x$ well inside a product neighborhood $N$
with $c(x)$ in the ambient neighborhood $N^*$ and $x_n\to x$, strictly monotonically from one side, on the same leaf of $N$ as $x$. We may assume that the channel in $N^*$ between the leaves $L$ and $L'$ of $N$ containing $x$ and $o(x)$, resp., is so narrow that $c(x_n)\subset N^*$ and $o(x_n)\in L'$. Now $c(x_n)\cap( c(x_m)\cup c(x))=\emptyset $ for $n\ne m$ (otherwise, $o(x_n)$ is ambiguous) so the sequence $o(x_n)$ converges monotonically on $L'$ to some $y$. If $y\ne o(x)$ then $y$ separates $o(x)$ from all $o(x_n)$ on $L'$. Let $z\in L'$ separate $y$ from $o(x)$ on $L'$. Then $o(z)$ must separate
$x$ from all $x_n$ on $L$ (otherwise, $c(z)\cap (c(x)\cup c(x_n))\ne \emptyset$ and $o(z)$ is ambiguous). But this is not possible, since $x_n\to x$ on $L$. Thus, $o$ is continuous.
\\
Suppose that $y=o(x)$ is also in $B$. Then $o:B\rightarrow B$, by continuity, and, since $o$ has a period two point, $o$ must have a fixed point, which it doesn't. So for each accessible branch $B$
there is an accessible branch $B'\ne B$ such that $o:B\rightarrow B'$ is a continuous surjection. 
Since $inf\{diam(c(x)):x\in B\}=0$, the corresponding continua $\hat{c}(x)\subset \hat U$ have diameters limiting on 0. Thus $B'$ must be adjacent to $B$; that is, the corresponding collections of accessible prime ends,
$\hat B$ and $\hat{B}'$ are the two branches of the stable manifold of a fixed inaccessible prime end.
\end{proof}


Suppose that $f$ leaves invariant the complimentary domain $U$ and also leaves invariant all the accessible unstable branches of periodic points accessible from $U$. Suppose that $B$ is an accessible unstable branch of an accessible periodic point. Given $x\in B$, let $B^-_x$ and $B^+_x$ denote the path components of $B\backslash \{x\}$. For each
$y\in B^-_x$ and $z\in B^+_x$, the continuum $c_{y,z}$ consisting of the union of $c(y)$, $c(z)$, the arc in $B$ from $y$ to $z$, and the arc in $B'$ from $o(y)$ to $o(z)$ separates $M$ into two components,
one of which, call it $U_{y,z}$, lies entirely in $U$. The continuum $C(y,z):=c(y,z)\cup U_{y,z}$ is a non-separating planar continuum, as is $C(x):=\cap_{y\in B^-_x,z\in B^+_x}C(y,z)$. $C(x)$ is just a thickened up version of $c(x)$ with the additional properties that $f(C(x))=C(f(x))$ and the collection $\{C(x):x\in B\}$ is upper semi-continuous. 
\\
Suppose that there are $n$ periodic, inaccessible prime ends $\mathbf{p}_0,\ldots,\mathbf{p}_{n-1}$, listed in cyclic order in $\partial \hat U$. For each $i\in \{0,\ldots,n-1\}$, let $\hat{B}_i$ be a branch of the stable manifold of $\mathbf{p}_i$ in  $\partial \hat U$ and let $B_i$, $i=0,\ldots,n-1$, be the corresponding accessible unstable branches of the accessible fixed points $p_i$ in $A$. We choose the $\hat{B}_i$ so that $o(B_i)\cap B_{i+1}=\emptyset$, $i\in \{0,\ldots,n-1\}$, subscripts taken mod($n$). For each
$(x_0,\ldots,x_{n-1})$ with $x_i\in B_i$,
$i=0,\ldots,n-1$, let $s(x_0,\ldots,x_{n-1}):=(\cup_{i\in \{0,\ldots,n-1\}}C(x_i))\cup (\cup_{i\in \{0,\ldots,n-1\}}
[o(x_i),x_{i+1}])$, where $[o(x_i),x_{i+1}]$ denotes the arc from $o(x_i)$ to $x_{i+1}$ in $W^u(p_{i+1})$, indices taken mod($n$). $M\backslash s(x_0,\ldots,x_{n-1})$ consists of two components, one of which, call it
$U_{(x_0,\ldots,x_{n-1})}$, is contained in $U$. Let $S_{(x_0,\ldots,x_{n-1})}:=s(x_0,\ldots,x_{n-1})\cup 
U_{(x_0,\ldots,x_{n-1})}$. Then 
$S_U:=\cap S_{(x_0,\ldots,x_{n-1})}$, the intersection being over all $(x_0,\ldots,x_{n-1})$, $x_i\in B_i$,
$i=0,\ldots.n-1$, is a continuum with $S_U\backslash U=\{p_0,\ldots,p_{n-1}\}$ and $f(S_U)=S_U$.
\\
Now for general $f$ let $m> 0$ be such that $f^m$ fixes all complementary domains of $A$ and the 
induced homeomorphism on prime ends fixes stable branches of periodic prime ends. Define
$C(x)$ for non-periodic accessible points $x$, and $S_U$ for each complementary domain $U$, as above, using $f^m$. Then $f(C(x))=C(f(x))$ for every $x$ accessible from $M\backslash A$ and $f(S_U)=S_{f(U)}$ for each complementary domain $U$.   We have:

\begin{lemma}\label{step1b}
The collection $\mathcal{C}$:=$\{C(x):x\in A$, $x$ is accessible from $M\backslash A$, and $x$ is not periodic under $f\}\cup \{S_U:U$ is a complimentary domain of $A\}\cup \{\{z\}: z\in A$ and $z$ is not accessible from $M\backslash A\}$ is an $f$-invariant upper semi-continuous  collection that covers $M$.
\end{lemma}


We would like for $\mathcal{C}$ to be a {\em decomposition} of $M$ - but if a singular point in
$A$ is accessible from complementary domains $U$ and $V$, then $S_U\cap S_V \ne \emptyset$.
We would also like for the elements of $\mathcal{C}$ to be non-separating continua - but $S_U$ will not have this property if some singular point in $A$ is accessible from $U$ by two different directions. We will
remedy these problems by splitting any singular point of $A$ that is accessible from $k$ distinct
directions (that is, has $k$ distinct accessible prime ends associated with it) into $k$ distinct points.  

\begin{lemma}\label{step2}
Let $A\subset M$ be a transitive expansive attractor of a homeomorphism $f:M\to M$ such that each domain of $M\setminus A$ is a topological disk.  Let $\mathcal{S}$ denote the collection of all singular points of $A$ that are accessible from more than one direction in $M\backslash A$.   Then there exists a surface $M'$, a homeomorphism $f':M'\ra M'$ with transitive expansive attractor $A'$, and a semi-conjugacy from $f'|_{A'}$ to $f|_A$ that is finite-to-1 on preimages of $\mathcal{S}$, 1-to-1 on preimages of $A\setminus \mathcal{S}$, and no point of $A'$ is accessible from more than one direction in $M'\backslash A'$. Moreover, every component of $M'\setminus A'$ is a topological disk.
\end{lemma}
\begin{proof}
For each complementary domain $U^j$, consider the components $U_0^j,\ldots,U_{n-1}^j$, $n=n(j)$, of
$U^j\backslash S_U^j$. On the circle of prime ends associated with $U_i^j$ there is a pair of accessible prime ends  $\hat{p}_i^j$ and $\hat{p}_{i+1}^j$ corresponding to accessible periodic points $p_i^j$ and $p_{i+1}^j$ and an arc $\gamma_i^j$ between $\hat{p}_i^j$ and $\hat{p}_{i+1}^j$ consisting of
prime ends having defining sequences of cross-cuts with endpoints in $S_U^j$. We compactify $M\backslash S_U^j$ with the arcs $\gamma_i^j$ in the obvious way.
If $U^1,\ldots,U^m$ are the complementary domains of $A$, let $M_0:=(M\backslash (\cup_{j\in\{1,\dots,m\}}S_U^j))\cup (\cup_{j\in \{1,\ldots,m\}, i\in \{1,\ldots,n(j)\}}\gamma_i^j)$ and let
$f_0:M_0\rightarrow M_0$ be given by $f_0(x)=f(x)$ if $x\in M\backslash (\cup_{j\in\{1,\dots,m\}}S_U^j)$ and
$f_0(\mathbf{x})=\hat {f}(\mathbf{x})$ if $\mathbf{x}\in \cup_{j\in \{1,\ldots,m\}, i\in \{1,\ldots,n(j)\}}\gamma_i^j$. Now for each $j$, sew  the boundary of a closed disk $D^j$ to $\cup_{ i\in \{1,\ldots,n(j)\}}\gamma_i^j$ to create a surface $M_1$. It is an easy matter to extend $f_0$ to a homeomorphism
$f_1:M_1\rightarrow M_1$ having $A$ as an expansive attractor (and $f_1|_A=f|_A$). In passing from
$f:M\rightarrow M$ to $f_1:M_1\rightarrow M_1$ we have replaced each $S_{U}^j$ with a more manageable $D^j$ which will make it easier to blow up periodic singular points that are accessible from
more than one direction.
\\
On the boundary of each disk $D^j$ there are points $x_r^j$, $r=1,\ldots,n(j)$, that are identified with the 
periodic points accessible from $U^j$ (we abuse notation to use $U^j$ to denote the complementary domain in $M_1\backslash A$ that meets $D^j$). We may assume that in the construction of $f_1$ we have 
created closed arcs $\tau^j_r\subset D^j$ with $\tau^j_r\cap \tau^j_k=\emptyset$ for each $j$ and each
$r\ne k\in \{1,\ldots,n(j)\}$, with $\tau^j_r\cap \partial D^j=\{x_r^j\}$, $\tau^j_r\setminus \{x_r^j\}\subset\mathring{D}^j$, and with $\{\tau^j_i:r=1,\ldots,n(j),j=1,\ldots,m\}$ invariant under $f_1$ (in particular, for
each $r$ and $j$, the endpoint of $\tau^j_r$ interior to $D^j$ is periodic).
\\
Let $\mathcal{S}$ denote the collection of all singular points of $A$ that are accessible from more than one direction in $M_1\backslash A$, and suppose that $p\in \mathcal{S}$ is accessible from $k=k(p)>1$ directions. Let $b_0,\ldots,b_{n-1}$ be the local unstable branches of $p$
listed, mod($n$), in cyclic order. The accessible branches occur in disjoint pairs $\{b_{i(s)},b_{i(s)+1}\}$,  $s=0,\ldots,k-1$. We choose the subscripting and the function $i$ so that $i(0)=0$ and $i:\{0,\ldots,k-1\}\to \{0,\ldots,n-1\}$ is increasing. For each $s\in \{0,\ldots,k-1\}$, let $U^{j(s)}$ be the complementary domain
of $A$ in $M_1$ from which $b_{i(s)}$ and $b_{i(s)+1}$ are accessible. There is then a unique $r(s)\in \{1,\ldots,n(j(s))\}$ so that $\tau^{j(s)}_{r(s)}$ accesses $p$ between $b_{i(s)}$ and $b_{i(s)+1}$. For
each $j\in \{1,\ldots.m\}$ let $T(j):=\{s: j(s)=j\}$ and let $W^j:=D^j\backslash (\partial D^j \cup (\cup_{s\in T(j)}\tau^j_{r(s)}))$. For each point $x\in \partial D^j \backslash (\cup_{s\in T(j)}\tau^j_{r(s)})$ there
is a unique accessible prime end $\hat x \in \partial \hat{W}^j$; for the interior (to $D^j$) endpoint $e(j,r(s))$ of $\tau ^j_{r(s)}$ there is a unique accessible prime end $\hat{e}(j,r(s))\in \partial \hat{W}^j$; and for each $s\in T(j)$ and for each $y\in \tau^j_{r(s)}\backslash \{e(j,r(s))\}$ there are exactly two accessible prime ends,
$\hat{y}^-$ and $\hat{y}^+$ in $\partial \hat{W}^j$, with $\hat{y}^-$ on the $b_{i(s)}$ side of $\tau^j_{r(s)}$
and $\hat{y}^+$ on the $b_{i(s)+1}$ side. In particular, there are accessible prime ends $\hat{x}^{j+}_{r(s)}$ and $\hat{x}^{j-}_{r(s)}$ corresponding to the endpoint $x^j_{r(s)}$ of $\tau^j_{r(s)}$ on $\partial D^j$ that is identified with $p$ in the construction of $M_1$.
\\
Now let $\bar{W}(p)$ be the quotient of the union of the disks $\hat{W}^{j(s)}$ of prime points and prime ends of $W^{j(s)}$, over $s\in \{0,\ldots,k-1\}$, in which $\hat{x}^{j(s)+}_{r(s)}$ is identified with $\hat{x}^{j(s+1)-}_{r(s+1)}$, $s+1$ being understood mod($k$). Let $X:=M_1\backslash(\{p\}\cup(\cup_{s=0,\ldots,k-1}\mathring{D}^{j(s)}))$. The space $X$ has $k$ ends: let $\hat{p}_s$, $s=0,\ldots,k-1$ denote the end of $X$ on which $b_{i(s)}$ and $b_{i(s)-1}$, $i(s)-1$ taken mod($k$), limit. Let $\bar{X}$ denote the end-compactification of $X$: $\bar{X}:=X\cup \{\hat{p}_0,\ldots,\hat{p}_{k-1}\}$. Let $Y$ denote the quotient of the union of $\bar{X}$ and $\bar{W}(p)$ 
in which every point $x\in \partial D^{j(s)}\backslash \cup_{t\in T(j(s))}\tau^{j(s)}_{r(s)}$ is identified with the corresponding point $\hat{x}\in \partial \hat{W}^{j(s)}$ for each $s=0,\dots,k-1$ and in which the point $\hat{x}^{j+}_{r(s)}=\hat{x}^{j-}_{r(s)}$ is identified with $\hat{p}_i$, $i=0,\ldots,k-1$. $Y$ is a surface with a single boundary component,  $\partial Y=\cup_{s=0,\ldots,k-1}\hat{\tau}^{j(s)}_{r(s)}$, where $\hat{\tau}^{j(s)}_{r(s)}$
is the arc $\hat{\tau}^{j(s)}_{r(s)}:=\{\hat{y}^-:y\in \tau^{j(s)}_{r(s)}\}\cup \{\hat{e}(j(s),s)\}\cup\{\hat{y}^-:y\in \tau^{j(s)}_{r(s)}\}$. Let $D(p)$ be a disk and sew $\partial D(p)$ to $\partial Y$ to make the surface $M_1'$.
\\
Let $A_1':=(A\backslash \{p\})\cup \{\hat{p}_0,\ldots,\hat{p}_{k-1}\}\subset M_1'$. The points $\hat{p}_i$ may still be singular in $A_1'$, but each is accessible only from the single complementary domain of $A_1'$ in $M_1'$ that contains $\mathring{D}(p)$ and the interiors of all the  $\hat{W}^{j(s)}$, $s=0,\ldots,k-1$.
If there is still a singular point $q$ in $A_1'$ that is accessible from more than one complementary domain of $A_1'$ in $M_1'$, we repeat the above process, replacing $A$ by $A_1'$, $M$ by $M_1'$, $p$ by $q$, and, if $j=j(s)$ for some $s$, $D^j$ by $\hat{W}^j$. It is important to note we already have the necessary arcs $\tau^j_i$ accessing $q$ from the various
complementary domains - we use these in all iterations of the above so that, when the splitting open process is finally completed, we know how to define the map.
\\
Let $M'$ denote the surface that results from splitting all points in $\mathcal{S}$ and let $A':=(A\backslash \mathcal{S})\cup (\cup_{p\in \mathcal{S},s\in \{0,\ldots,k(p)-1\}} \{\hat{p}_s\})$. We define $f':M'\to M'$
as follows. If $p\in \mathcal{S}$ then $f(p)=q$ for some $q\in \mathcal{S}$. Then $k(q)=k(p)$ and for each $s\in \{0,\ldots,k(p)-1\}$ there is a unique $t\in \{0,\ldots,k(q)-1\}$ such that $f(\tau^{j(s)}_{r(s)})=\tau^{j(t)}_{r(t)}$ (in fact $s-t$ is constant mod($k(p)$)): define $f'(\hat{p}_s):=\hat{q}_t$. For $y\in \tau^{j(s)}_{r(s)}\backslash \{e(j(s),r(s))\}$, $z=f(y)\in \tau^{j(t)}_{r(t)}\backslash \{e(j(t),r(t))\}$: define $f'(\hat{y}^-):=\hat{z}^-$,  $f'(\hat{y}^+):=\hat{z}^+$, and $f'(\hat{e}(j(s),r(s)):=\hat{e}(j(t),r(t))$. (Recall that $\hat{e}(j,r)$ denotes the accessible prime end in $\partial \hat{W}^j$ corresponding to the interior endpoint of $\tau_r^j$.)
This defines $f'$ on $\cup_{p\in \mathcal{S}} \partial D(p)$. We extend $f'$ to a homeomorphism of
$\cup_{p\in \mathcal{S}} D(p)$ in such a way that $\cup_{p\in \mathcal{S}, s\in \{0,\ldots,k(p)-1\}} \{\hat{p}(s)\}$ is locally attracting in $\cup_{p\in \mathcal{S}} D(p)$. All other points in $x\in M'$ are points of $M$ and for these we define $f'(x):=f(x)$. The result is that $f':(M',A')\rightarrow (M',A')$ is an extension
of  $f:(M,A)\rightarrow (M,A)$, no point of $A'$ is accessible from more than one direction in $M'\backslash A'$, $A'$ is an expansive attractor in $M'$ under the homeomorphism $f'$, and
$(A,f)$ is recovered from $(A',f')$ by the quotient map that identifies $\hat{p}_0,\ldots,\hat{p}_{k(s)-1}$
with $p$ for each $p\in \mathcal{S}$, and is otherwise one-to-one. \\

Now it may be the case that a some complementary domains of $A'$ in $M'$ are not open topological disks (this will happen, for example, if some $p\in A$ is accessible from two different directions in a single complementary domain of $A$ in $M$). If this is the case, apply Lemma~\ref{step0}.
\end{proof}
\begin{remark}\label{remark1} 
For the $f':(M',A')\to (M',A')$ constructed in Lemma \ref{step2}, 
each of the decomposition elements $S_U$, $U$ a complementary domain of $A'$ in $M'$, is a closed topological disk that is periodic under $f'$. By modifying $f'$ only on the interiors of the $S_U$, we may arrange that if $S_U$ has period $n$, then there is a periodic point $p_U\in\mathring{S}_U$ of period $n$.
\end{remark}
\begin{prop}\label{step3}
Suppose that $f:M\to M$ is a surface homeomorphism with transitive expansive attractor $A$, that every
component of $M\setminus A$ is a topological disk, and that no point of $A$ is accessible from $M\setminus A$ by two inequivalent directions. The collection $\mathcal{C}$ of Lemma~\ref{step1b} is then an$f$-invariant upper semi-continuous decomposition of $M$ into a surface $\tilde{M}$ homeomorphic with $M$.
\end{prop}

\begin{proof}
If $U$ and $V$ are distinct complementary domains of $A$, then $S_U\cap S_V=\emptyset$ since no points of $A$ are accessible from both $U$ and $V$. Hence $\mathcal{C}$ is a decomposition of $M$. Furthermore, no point of $A$ is accessible from two different directions in the same complementary domain, so each $S_U$ is a non-separating planar continuum. Thus $\mathcal{C}$ is an upper semi-continuous decomposition of $M$ into non-separating planar continua that is invariant under $f$. Moore's theorem (\cite{Moore}) asserts that the decomposition space $\tilde M$ is a surface homeomorphic with $M$.
\end{proof}

Note that each decomposition element of $\mathcal{C}$ meets $A$ in a non-empty and finite set. Let
$\pi:A\to \tilde{M}$ be the quotient map that takes $x\in A$ to the $C\in \tilde{M}$ for which $x\in C$.
Then $\tilde{f}\circ \pi=\pi \circ f$ and to prove Theorem \ref{classificationtheorem} it remains to show that $\tilde{f}$ is pseudo-Ansov.

\section{Double Covers}\label{doublecover}
Were the $\tilde{f}:\tilde{M}\to \tilde{M}$ constructed in the preceding section expansive, it would be immediate that $\tilde{f}$ is pseudo-Anosov. However, if $U$ is a complementary domain from which only one periodic point $p$ of $f$ in $A$ is accessible, then the unstable foliation of $\tilde{M}$ induced by $A$ may have a 1-pronged singularity at $\pi(p)$: as in the Plykin example, $\tilde{f}$ would fail to be expansive. We will fix this by passing to a branched double cover of $\tilde{M}$. That $\tilde{f}$ is pseudo-Anosov will follow from expansivity of the lift of $\tilde{f}$ to the double cover.  
\\
The argument we will give for expansivity of the lift of $\tilde{f}$ will be based on viewing the double cover of $\tilde{M}$ as the decomposition space of an expansive attractor that double covers $A$, and this argument will require that for each
complementary domain $U$ of the double cover of $A$ there are at least two periodic points accessible from $U$.  The double covering will be branched exactly over the odd-pronged singularities. If $p$ is the only periodic point accessible from $U$ and $p$ is either nonsingular or an even-pronged singularity, then $\pi(p)$ will be an odd-pronged singularity in $\tilde{M}$, which will be unwrapped by the double cover. That is, $U$ will be unwrapped into a domain from which $p$ and its twin are accessible. If $p$ is the sole fixed point accessible from $U$ and is odd-pronged, double covering will fail to unwrap $U$. In this case there is an inaccessible branch of $W^u(p)$ which we unzip to create a second fixed point accessible from $U$ (the unzipping process is described in an appendix).
\\
Given a closed surface $S$ (like $\tilde{M}$) that has a one-dimensional foliation $\mathcal{F}$ whose
(finitely many) singularities $\mathcal{S}$ are of the finite-branch variety (like the unstable foliation of $\tilde{M}$), 
there is a double cover of the surface, $\eta:\tilde{S}\to S$, branched over the odd-pronged singularities
of $\mathcal{F}$, with the property that the pull-back $\tilde{\mathcal{F}}$ of $\mathcal{F}$ has only
even-pronged singularities. Moreover, any homeomorphism $g:S\to S$ that preserves $\mathcal{F}$
lifts to a homeomorphism $\tilde{g}:\tilde{S}\to \tilde{S}$ that preserves $\tilde{\mathcal{F}}$. 
The double cover is constructed as follows. Given an $n$-pronged singularity $p$ of $\mathcal{F}$, let $i(p):=(2-n)/2$. By the Poincar\'{e}-Hopf formula, $\sum_{p\in \mathcal{S}}i(p)=\chi(S)$ is an integer, thus there is an even number, say $2k$, of odd-pronged singularities in $\mathcal{F}$. 

Let $\{U_{\beta}\}$ be an open cover of $S\setminus \mathcal{S}$ for which there are homeomorphisms
$\Phi_{\beta}:(0,1)\times (0,1)\to U_{\beta}$ so that $\Phi_{\beta}^*(\mathcal{F})$ is the horizontal foliation of
$(0,1)\times (0,1)$ and assume that $U_{\alpha}\cap U_{\beta}$ is connected for each $\alpha$ and $\beta$.  Let $\mathcal{D}:=\cup_{\beta}(U_{\beta}\times \{\beta\} \times \{-1\}\cup U_{\beta}\times \{\beta\} \times \{1\})$
be the disjoint union of two copies of each $U_{\beta}$ and let $\sim$ be the equivalence relation on
$\mathcal{D}$ defined by $(x,\alpha,i)\sim (y,\beta,j)$ iff: $x=y$, $i=j$, and $\Phi_{\alpha}^{-1}\circ \Phi_{\beta}$ is increasing in the horizontal coordinate (where defined); or $x=y$, $i=-j$, and $\Phi_{\alpha}^{-1}\circ \Phi_{\beta}$ is decreasing in the horizontal coordinate (where defined). $\mathcal{D}/\sim$ is 
a surface and $\eta:\mathcal{D}/\sim \to S\setminus \mathcal{S}$ given by $\eta([x,\beta,i])=x$ is a double
cover. If $m$ is the number of even-pronged singularities in $\mathcal{S}$, then $\mathcal{D}/\sim$ has
$2m+k$ ends. Let $\tilde{S}$ be the end-compactification of $\mathcal{D}/\sim$. The double cover then
extends to a branched double cover $\eta:\tilde{S}\to S$, branched over each odd-pronged singularity in
$\mathcal{S}$. The pull-back $\tilde{\mathcal{F}}:=\eta^*(\mathcal{F})$ is orientable; in particular, an end of $\mathcal{D}/\sim$ lying over a $(2n+1)$-pronged singularity of $\mathcal{F}$ is a $(4n+2)$-pronged singularity of $\tilde{\mathcal{F}}$. Let $\tilde{\mathcal{S}}:=\tilde{S}\setminus (\mathcal{D}/\sim)=\eta^{-1}(\mathcal{S})$.

There are any number of ways to ``unwrap" an odd-pronged singularity with a branched covering map. The reason we use the orientation double cover is that we can be sure that foliation preserving homeomorphisms lift. To see that this is the case, consider an element $[\gamma]$ in the fundamental
group $\pi_1(S\setminus \mathcal{S})$. There are then $\beta_0,\ldots,\beta_{n-1}$ and $0=t_0<t_1<\cdots<t_n=1$ so that $\gamma([t_{i},t_{i+1}])\subset U_{\beta_i}$ for $i=0,\ldots,n-1$. For each $i\in \{0,\ldots,n\}$, let $r(i)=1$ if $\Phi_{\beta_{i+1}}^{-1}\circ \Phi_{\beta_i}$ is increasing in the horizontal coordinate, and let $r(i)=-1$ if $\Phi_{\beta_{i+1}}^{-1}\circ \Phi_{\beta_i}$ is decreasing in the horizontal coordinate (we take subscripts mod($n$), so that $\beta_n=\beta_0$). Let $r(\gamma):=\prod_{i=0,\ldots,n}r(i)$. Then $r(\gamma)$ is independent of the choice of the $\beta_i$ and is stable under perturbation of $\gamma$ so we have a well-defined $r:\pi_1(S\setminus \mathcal{S})\to \{-1,1\}$, which is clearly a group homomorphism. One sees directly from the construction of $\tilde{S}$ that $\gamma$ lifts to a loop in $\tilde{S}\setminus \tilde{\mathcal{S}}$ if and only if $r(\gamma)=1$. That is, $\eta_*(\pi_1(\tilde{S}\setminus \tilde{\mathcal{S}}))=ker(r)$.

Now, if $g:S\to S$ is a homeomorphism that preserves $\mathcal{F}$ and $\gamma$ is a loop in
$S\setminus \mathcal{S}$, then $r(g\circ \gamma)=r(\gamma)$. Thus $(g\circ \eta)_*(\pi_1(\tilde{S}\setminus \tilde{\mathcal{S}})\subset \eta_*(\pi_1(\tilde{S}\setminus \tilde{\mathcal{S}})$ and it follows
that $g:S\setminus \mathcal{S}\to S\setminus \mathcal{S}$ lifts to a homeomorphism $\tilde{g}:\tilde{S}\setminus \tilde{\mathcal{S}}\to \tilde{S}\setminus \tilde{\mathcal{S}}$, which clearly extends
to $\tilde{g}:\tilde{S}\to \tilde{S}$.

Suppose that $f:M\to M$ is a homeomorphism having expansive attractor $A$ with the properties:
all complementary domains of $A$ are topological disks; all periodic points in $A$ accessible from
$M\setminus A$ are fixed by $f$ and none of these fixed points is accessible from inequivalent directions in $M\setminus A$; if $p$ is a fixed point of $A$ accessible from the complementary domain $U$ and there are no other fixed points in $A$ accessible from $U$, then $W^u(p)$ has an even number of branches; and in the upper semi-continuous decomposition $\mathcal{C}$ of $M$ as constructed in Section ~\ref{usc}, each of the decomposition elements $S_U$, $U$ a complementary domain is a closed topological disk having a point $p_U\in \mathring{S}_U$ fixed by $f$. (Recall that $S_U$ is the decomposition element containing the fixed points in $A$ accessible from $U$.) To avoid proliferation of tilde's, let
$f_1:M_1\to M_1$ denote the induced homeomorphism on the decomposition space $M_1$.
Let $\eta_1:\tilde{M}_1\to M_1$ be
the orientation double of the foliation $\mathcal{F}$ of $M_1$ induced by $A$ constructed above and let $\tilde{f}_1:\tilde{M}_1\to \tilde{M}_1$ be a lift of $f_1$.

\begin{prop}\label{doubcov}
There is a surface $\tilde{M}$, a homeomorphism $\tilde{f}:\tilde{M}\to \tilde{M}$ with transitive expansive attractor $\tilde{A}$ and a branched double cover $\eta:\tilde{M}\to M$ with the properties:
\begin{enumerate}
\item $\eta \circ f=\tilde{f}\circ \eta$;
\item $\eta(\tilde{A})=A$;
\item all complementary domains of $\tilde{A}$ in $\tilde{M}$ are topological disks;
\item no point of $\tilde{A}$ is accessible from the complement by inequivalent directions;
\item if $p\in \tilde{A}$ is a periodic point of $\tilde{f}$ accessible from a complementary domain $U$
of $\tilde{A}$, then there is another periodic point $q\ne p$, $q\in \tilde{A}$, of $\tilde{f}$ that is also accessible from $U$.
Moreover,
\item $\tilde{M}_1$ is the decomposition space of $\tilde{M}$, as in Proposition \ref{step3}, and 
$\tilde{\pi} \circ \tilde{f}_1=\tilde{f}\circ \tilde{\pi}$, with the decomposition map $\tilde{\pi}:\tilde{M}\to\tilde{M}_1$ a lift of the decomposition map $\pi:M\to M_1$. 
\end{enumerate}
\end{prop}
\begin{proof}
Let $\mathcal{S}$ denote the set of singularities of the unstable foliation $\mathcal{F}$ of $M_1$: note that each $S_U$, $U$ a complementary domain from which only one fixed point in $A$ is accessible, is an odd-pronged singularity in $\mathcal{S}$. Let $\mathcal{D}$ be the disjoint collection of disks $U_{\beta}\times \{\beta\} \times \{i\}$, and $\sim$ the equivalence relation, 
as in the above construction of $\tilde{M}_1$.

Let  $\mathcal{D}_0:=\cup_{\beta}((U^0_{\beta}\times \{\beta\} \times \{-1\})\cup(U^0_{\beta}\times \{\beta\} \times \{1\}))$, where $U^0_{\beta}:=\cup_{C\in U_{\beta}}C$, and define the equivalence relation $\sim_0$ on $\mathcal{D}_0$ by $(x,\alpha,i)\sim_0(y,\beta,j)$ iff $x=y$ and $(C,\alpha,i)\sim(C,\beta,j)$, where $x\in C\in \mathcal{C}$. Let $\tilde{M}_0:=\mathcal{D}_0/\sim_0$. (So $\tilde{M}_1\setminus \tilde{\mathcal{S}}$ is a u.s.c.-decomposition of $\tilde{M}_0$.) Let $\mathcal{S}_0:=\cup_{C\in \mathcal{S}}C\subset M$. Then $\eta_1:\tilde{M}_0\to M\setminus \mathcal{S}_0$ is a double cover and the restriction $f:M\setminus \mathcal{S}_0\to M\setminus \mathcal{S}_0$ lifts to a homeomorphism $\tilde{f}_0:\tilde{M}_0\to
\tilde{M}_0$ defined by $\tilde{f}_0([(x,\beta,i)])=[(f(x),\alpha,j)]$ provided $\tilde{f}_1([C,\beta,i)]=[(f(C),\alpha,j)]$,where  $x\in C\in \mathcal{C}$.

The singularities $C\in \mathcal{S}$ come in two forms: $C=\{p\}$ for some singularity $p$ of $A$, $p$ inaccessible from $M\setminus A$, or $C=S_U$, $U$ a complementary domain of $A$ in $M$. For each 
of these singularities $C$, let $B_C$ be a closed topological disk in $M$ with $C\subset \mathring{B}_C$. We may choose these disks to be pairwise disjoint: $B_C\cap B_{C'}=\emptyset$ for $C\ne C'\in \mathcal{S}$. Let $\mathcal{S}^e$ and $\mathcal{S}^o$ be the collections of even- , resp., odd-pronged singularities of $\mathcal{F}$. If $C\in \mathcal{S}^e$, the inclusion $i_C:B_C\setminus C\to M\setminus \mathcal{S}_0$ lifts in two ways to embeddings $\tilde{i}^{\pm}_C:B_C\setminus C\to \tilde{M}_0\setminus \tilde{\mathcal{S}}_0$. We take the disks $B_C$ small enough so that the images of $\tilde{i}^+_C$ and $\tilde{i}^-_C$ are disjoint.

If $C\in \mathcal{S}^o$, let $\delta_C:B_C\to B_C$ be a branched double cover, branched over $p$ if $C=\{p\}$ or over $p_U$ if $C=S_U$. Then the restriction of $\delta_C$ to 
$B_C\setminus C$, followed by the inclusion of $B_C\setminus C$ into $M\setminus \mathcal{S}_0$
lifts to an embedding $\tilde{\delta}_C:B\setminus C\to \tilde{M}\setminus \tilde{\mathcal{S}}_0$.

Let $\tilde{M}:=\tilde{M}_0\cup(\cup_{C\in \mathcal{S}^e}(B_C\times \{-1\}\cup B_C\times \{1\}))\cup(\cup_{C\in \mathcal{S}^o}B_C)/\bowtie$, with $\bowtie$ identifying $(x,\pm1)\in (B_C\setminus C)\times \{\pm1\}$ with $\tilde{i}^{\pm}_C(x)$ and, for $c\in \mathcal{S}^o$, $x\in B_C\setminus C$ with $\tilde{\delta}_C(x)$. Now $\eta_1: \tilde{M}_0\to M\setminus \mathcal{S}_0$ extends naturally to a branched double cover $\eta:\tilde{M}\to M$ and $\tilde{f}_0:\tilde{M}_0\to \tilde{M}_0$ extends to a homeomorphism $\tilde{f}:\tilde{M}\to \tilde{M}$ that is a lift of $f$. It is straightforward to check that $\tilde{A}:=\eta^{-1}(A)$ is a transitive expansive attractor for $\tilde{f}$. Moreover, if $U$ is a complementary domain of $\tilde{A}$ in $\tilde{M}$, then $\tilde{f}$ has at least two periodic points in $\tilde{A}$ accessible from $U$.


\end{proof}

\section{Expansivity of the double cover and proof of the classification theorem}\label{exp}

We assume that $\tilde{f}:\tilde{M}\to \tilde{M}$ has transitive expansive attractor $\tilde{A}$ with constant $\alpha$, that every complementary domain of $\tilde{A}$ in $\tilde{M}$ is an open topological disk, and that no point of $\tilde{A}$ is accessible from inequivalent directions in $\tilde{M}\setminus \tilde{A}$.
Let $\ACC$ denote the points of $\tilde{A}$ that are accessible from $\tilde{M}\setminus \tilde{A}$ and let $\ACC_P$ denote 
the set of points in $\ACC$ that are periodic. Let $\mathcal{C}$ be the upper semi-continuous decomposition of $\tilde{M}$ as in the Section \ref{usc} and let $\sim$ be the equivalence on $\tilde{A}$: $x\sim y$ iff $x$ and $y$ are in the same element of $\mathcal{C}$ (that is, $\tilde{\pi}(x)=\tilde{\pi}(y)$). We let $[x]$ denote the $\sim$ equivalence class of $x$. Then $[x]=\{x\}$ if $x\in \tilde{A}\setminus \ACC$; $[x]=\{x,o(x)\}$ if $x\in \ACC \setminus \ACC_P$; and $[x]=\{p:p\in \ACC_P$ and $x$ and $p$ are accessible from the same complementary domain of $\tilde{A}\}$ if $x\in \ACC_P$. We assume that for $x\in \ACC_P$, $[x]$ is not a singleton. Thus if $x\in (W^u(p)\setminus \{p\})\cap \ACC$ for some $p\in \ACC_P$, then $o(x)\in (W^u(q)\setminus \{q\})\cap \ACC$ for some $q\ne p$, $q\in \ACC_P$. Let $\tilde{M}_1:=\{[x]:x\in \tilde{A}\}$, with the quotient topology, and let $\tilde{f}_1:\tilde{M}_1\to \tilde{M}_1$ by $\tilde{f}_1([x]):=[\tilde{f}(x)]$. 
\begin{lemma}\label{expansive} Under the above assumptions, $\tilde{f}_1$ is expansive.
\end{lemma}
\begin{proof} As a homeomorphism is expansive if and only if any of its nonzero powers is expansive, we may assume that all elements of $\ACC_P$ are fixed by $\tilde{f}$ and that every branch of $W^u(p),p\in \ACC_P,$ is invariant under $\tilde{f}$.
\\
Given a complementary domain $U$ of $\tilde{A}$ in $\tilde{M}$, let $p_0,\ldots,p_{k-1}$ be the elements of
$\ACC_P$ accessible from $U$, indexed in cyclic order (that is, as their corresponding prime ends occur on the circle of prime ends associated with $U$). Let $b^{\pm}_i$ denote the two accessible branches of
$W^u(p_i)\setminus \{p_i\}$ with $o(b^+_i)=b^-_{i+1}$, $i=0,\ldots,k-1$, subscripts taken mod($k$).
Choose points $y^{\pm}_i\in b^{\pm}_i$ so that $o(y^+_i)=y^-_{i+1}$ for each $i$, let
$[y_i^+,\tilde{f}(y_i^+)]$ denote the closed arc in $b^+_i$ with endpoints $y_i^+ $ and $\tilde{f}(y_i^+)$, and similarly for $[y_i^-,\tilde{f}(y_i^-)]$. For each $N\in \mathbb{N}$ let $K_N(p_i):=\{p_i\}\cup (\cup_{n\le N}\tilde{f}^n([y_i^-,\tilde{f}(y_i^-)]\cup [y_i^+,\tilde{f}(y_i^+)]))$ for each $i$ and let $K_N(U):=\cup_{i=0,\ldots,k-1}K_N(p_i)$. We do the above for each complementary domain and let $K_N:=\cup_U K_N(U)$, the union being over all complementary domains $U$ of $A$ in $M$. The following definitions are made in the context of the metric space ($\tilde{A}$,$d$).

\begin{enumerate}
\item {\bf Defining $D$:} Let $D=\min\{d(p, q) : p\ne q,p, q\in\ACC_P\}$.
\item {\bf Defining $N_1$:} Choose $N_1$ large enough so that if $x\in\ACC\setminus K_{N_1}$, then $d(x, o(x))<\min\{D/3, \alpha/3\}$.
\item {\bf Defining $N_2$:} Let $N_2<N_1$ be small enough so that, for $p\ne q, p,q\in \ACC_P$, if 
$x\in K_{N_2}(p)$ and $y\in K_{N_2}(q)$ then $d(x,y)\ge 2D/3$. 
\item {\bf Defining $\delta$:}  Let $\delta>0$ be small enough so that if $p, q\in\ACC_P, p\not= q$ then\\
  (i) if $x\in \tilde{f}^{-1}(B_\delta(p))\cup B_\delta(p)\cup \tilde{f}(B_\delta(p))$ and $y\in \tilde{f}^{-1}(B_\delta(q))\cup B_\delta(q)\cup\tilde{f}(B_\delta(q))$, then $d(x, y)>2D/3$;\\
  (ii)  $B_{\delta}(p)\cap K_{N_1}=B_{\delta}(p)\cap \tilde{f}^2(K_{N_1})$; and\\
  (iii) if $x\in B_{2\delta}(p)\setminus \{p\}$ then $o(x)\in K_{N_2-2}$. 
\item  {\bf Defining $B(p)$ and $B$:}  For each $p\in\ACC_P$, let $B(p)\subset B_\delta(p)$ be a  product    
(or singular product) neighborhood of $p$ and let $B=\cup_{p\in\ACC_P}B(p)$.  
\item {\bf Defining $N_3$:} Let $N_3<N_2$ be small enough so that $K_{N_3}\subset B$.
\item  {\bf Defining $\eta$:} Let $\eta>0$ be small enough so that if $x,y\in K_{N_1}$ and $d(x, y)<\eta$, then $d(\tilde{f}^k(x), \tilde{f}^k(y))<\alpha$ for $k\le 2N_1-N_2-N_3+1$.
\item  {\bf Defining $\epsilon$:} Choose $\epsilon>0$ small enough so that\\
 (i) $\epsilon< $min$\{\alpha/6,D/6,\delta,\eta/2\}$;\\
 (ii) if $x,y\in K_{N_2+1}\setminus \tilde{f}^{-1}(K_{N_2})$ and $d(x,y)<\epsilon$, then $d(o(x),o(y))<\eta/2$; \\
 (iii) if $d(x,K_{N_1})<\epsilon$ then $\tilde{f}^{N_3-N_1}(x)\in B$;\\
 (iv) if $x\in \tilde{f}^{-1}(B)\setminus B$ then $d(x,K_{N_1})>\epsilon$;\\
 (v) if $x\in K_{N_2}\setminus K_{N_2-1}$ and $y\in K_{N_2-2}\cup (K_{N_1}\setminus K_{N_2+1})$ then $d(x,y)>\epsilon$.
 
\end{enumerate}

For each $[z]\in\tilde A$, let $U_{[z]}=\{[x]\in\tilde A : [x]\subset B_\epsilon([z])\}$.  Then $\{U_{[z]} : [z]\in\tilde A\}$ is an open cover of $\tilde A$.  We will establish expansivity of $\tilde f_1$ by showing that given $[x]\not= [x']$, there is an $n\in\mathbb{Z}$ so that for no $[z]\in\tilde A$ is it the case that $\tilde f_1^n([x]), \tilde f_1^n([x'])\in U_{[z]}$.  Let $x_n=\tilde{f}^n(x)$ and $x'_n=\tilde{f}^n(x')$ and suppose that for each $n$ there is $z_n$ such that $[x_n], [x'_n]\in U_{[z_n]}$. Let $z_n,z'_n\in [z_n]$ be such that $d(x_n,z_n)<\epsilon$ and $d(x'_n,z'_n)<\epsilon$.

Since $\alpha$ is an expansivity constant for $\tilde f$ on $A$, there is an $n$ such that $d(x_n, x'_n)>2\alpha /3$.  Without loss of generality, we take $n$ to be zero so that we assume $d(x_0, x'_0)>2\alpha /3$.  Then from choices (8i) and (2) above, it must be the case that $z_0,z'_0\in K_{N_1}$ so that $d(x_0, K)<\epsilon$ and $d(x'_0, K)<\epsilon$. Choice (8iii) then implies that $x_{N_3-N_1}, x'_{N_3-N_1}\in B$. 
We consider three  exhaustive cases on the disposition of  $x_{N_3-N_1}$ and $ x'_{N_3-N_1}$ in $B$ and show that each leads to a contradiction.\\

{\bf Case 1: At least one of  $x_{N_3-N_1}$ or $x'_{N_3-N_1}$ is not in $\displaystyle\bigcup_{p\in\ACC_P}W_\delta^u(p)$.}  

Then there must exist an $n<N_3-N_1$ so that $x_{n}\in \tilde{f}^{-1}(B(p))\setminus B(p)$ and $x'_{n}\in \tilde{f}^{-1}(B(p'))\cup B(p')$ (or the same with the roles of $x$ and $x'$ reversed). Then, by condition (8iv), $d(x_{n}, K_{N_1})>\epsilon$ so that $z_n\not\in K_{N_1}$. Then $d(z_n,z'_n)<D/3$ by (2). But  choice (4i) for $\delta$ implies $d(x_{n}, x'_{n})>2D/3$, so $d(z_n,z'_n)> 2D/3-2\epsilon> D/3$, by (8i), a contradiction.\\

{\bf Case 2: Each of  $x_{N_3-N_1}$ and $x'_{N_3-N_1}$ is in $\displaystyle\bigcup_{p\in\ACC_P}W_\delta^u(p)$ and exactly one of them is not in $K_{N_1}$. }\\

Assume $x_{N_3-N_1}\in W^u_{\delta}(p)\setminus K_{N_1}$ and $x'_{N_3-N_1}\in W^u_{\delta}(p')\cap K_{N_1}$. Let $n_1:=inf\{n\ge N_3-N-1:x_n\not\in B\}$ and $n'_1:=inf\{n\ge N_3-N-1:x_n'\not\in K_{N_2-1}\}$. If $n_1\le n'_1$ then $d(x_{n_1},K_{N_1})>\epsilon$ and $d(x_{n_1},x'_{n_1})>2D/3$ and
the argument is as in Case 1. If $n'_1<n_1$ then $x'_{n'_1}\in K_{N_2}$. Also, since $d(x_{n'_1},x'_{n'_1})>2D/3$, $z'_{n'_1}\in K_{N_1}$, so, by (8v), $z'_{n'_1}\not\in K_{N_2-2}$. But then by (4iii),
$z_{n'_1}=o(z'_{n'_1})\not\in B_{2\delta}(\ACC_P)$. But $x_{n'_1}\in B_{\delta}(\ACC_P)$ and $\epsilon<\delta$ so $d(x_{n'_1},z_{n'_1})>\epsilon$, a contradiction.\\

{\bf Case 3: Both $x_{N_3-N_1}$ and $y_{N_3-N_1}$ are in $K_{N_1}$.}\\

At least one of these points is not in $\ACC_P$, say $x_{N_3-N_1}$.  Then there is $n$ so that $x_n\in K_{N_2}\setminus \tilde{f}^{-1}(K_{N_2})$ and $x'_n\in K_{N_1}$. Then $z_n\in K_{N_2+1}\setminus K_{N_2-1}$, by (8v), and, by (8ii), $d(o(x_n),o(z_n))=d(o(x_n),z'_n)<\eta/2$. Since $d(z'_n,x'_n)<\epsilon <\eta/2$,
$d(o(x_n),x'_n)<\eta$. Since $[x_n]\ne [x'_n]$, $o(x_n)\ne x'_n$, and there must be $k$ such that $d(\tilde{f}^k(o(x_n),\tilde{f}^k(x'_n))>\alpha$. By (7), such $k$ must be greater than $N:=2N_1-N_2-N_3+1$. Since $x_n\in K_{N_2}\setminus K_{N_2-1}$, $x_{n+k}\not\in K_{N_1}$ so $d(x_{n+k},o(x_{n+k})<\alpha/3$, by (2).
Thus $d(x_{n+k},x'_{n+k})>2\alpha/3$. Now, as in the preamble to Case 1, replacing $x_{n+k}$ by
$y_0$ and $x'_{n+k}$ by $y'_0$, we have $y_{N_3-N_1},y'_{N_3-N_1}\in B$. Since $k+N_3-N_1\ge
N_1-N_2+1$ and $x_n\in K_{N_2}\setminus K_{N_2-1}$, $y_{N_3-N-1}=x_{n+k+N_3-N_1}\not\in K_{N_1}$. This puts us in Case1 or Case 2.
\end{proof}
\begin{proof} (Theorem ~\ref{classificationtheorem})\\
Suppose that $A$ is a transitive expansive attractor for the homeomorphism $f:M\to M$ of the compact surface $M$. We may assume that $A$ and $M$ are connected. If $A=M$ then $A$ is trivially derived from pseudo-Anosov. Otherwise, there is a surface homeomorphism $f_1:M_1\to M_1$ with transitive expansive attractor $A_1$ and a map $\pi_1:A_1\to A$ that semi-conjugates $f_1\mid _{A_1}$ with
$f\mid_A$. All complementary domains of $A_1$ in $M_1$ are topological disks, no point of $A_1$ is accessible from the complement by two inequivalent directions, and $\pi_1$is one-to-one except on the pre-images of singularities of $A$ which are accessible from 
multiple inequivalent directions (lemmas ~\ref{step0} and ~\ref{step2}). There is then an upper semi-continuous decomposition of $A_1$ into 
a surface $M_2$ and a homeomorphism $f_2:M_2\to M_2$ induced by $f_1\mid_{A_1}$ (Proposition ~\ref{step3}). As any two pre-images of a single point under $\pi_1$ lie in the same decomposition element, there is a map $\pi:A\to M_2$ that semi-conjugates $f\mid_A$ with $f_2$. There is a branched covering $\eta:\tilde{M}_2\to M_2$ and a lift $\tilde{f}_2$ of $f_2$ to the surface $\tilde{M}_2$ that is expansive (Proposition \ref{doubcov} and Lemma ~\ref{expansive}). By \cite{hir}, \cite{lew}, $\tilde{f}_2$ is pseudo-Anosov and it follows that $f_2$ is also pseudo-Anosov
(with expanding and contracting foliations pushed down from $\tilde{M}_2$). With $f_2=g$ and $M_2=S$, the conditions for a derived from pseudo-Anosov attractor are satisfied by $f$ and $A$.
\end{proof}

\section{Appendix: unzipping a ray}

Suppose that $f:M\to M$ is a homeomorphism of the surface $M$ and $R\subset M$ is a ray: that is, $R$ is the image of a continuous, one-to-one map $\phi:\mathbb{R}^+\cup \{0\}\to M$. We assume that $R$ is invariant under $f$ and expanding in the sense that $\phi^{-1}\circ f^n\circ \phi(t)\to \infty$, as $n\to \infty$, for all $t>0$. Let $p=\phi(0)$ be the endpoint of $R$. We also assume that there are closed arcs
$A^-$ and $A^+$ in $M$ with the properties: $A^-\cap A^+=\{p\}=R\cap (A^+\cup A^-)$, and $A^-$ and $A^+$ are invariant under $f$.
Let $\Psi:\mathbb{D}^2\to M$ be an embedding with: $\Psi((0,0))=p$; $\Psi(\{(0,y):0\le y\le 1\})=\phi([0,r])$
for some $r>0$; $\Psi(\{(-t\sqrt{2}/2,-t\sqrt{2}/2):0\le t \le 1\})=A^-$; and  $\Psi(\{(t\sqrt{2}/2,-t\sqrt{2}/2):0\le t \le 1\})=A^+$. We further assume that there is an orientation preserving homeomorphism $\beta:[0,1]\to [0,1]$  for which 
$0$ and $1$ are attracting, and an orientation preserving homeomorphism $\gamma:[0,1]\to [0,1]$ for which $0$ is attracting 
, with the property: $\Psi^{-1}\circ f\circ \Psi (s(-t\sqrt{2}/2,-t\sqrt{2}/2)+(1-s)(t\sqrt{2}/2,-t\sqrt{2}/2))=\beta(s)(-\gamma(t)\sqrt{2}/2,-\gamma(t)\sqrt{2}/2)+(1-\beta(s))(\gamma(t)\sqrt{2}/2,-\gamma(t)\sqrt{2}/2)$,
for all $s,t \in [0,1]$. We make some definitions:
\begin{itemize}

\item $q:=\Psi^{-1}\circ f^{-1}\circ \Psi((0,1))$, 
\item $\Upsilon:= \{(-t\sqrt{2}/2,-t\sqrt{2}/2):0\le t \le 1\}\cup \{(t\sqrt{2}/2,-t\sqrt{2}/2):0\le t \le 1\}\cup \{tq:0\le t \le 1\}\subset \mathbb{D}^2$, 
\item $T^{-}:=\{s((1-t)(-\sqrt{2}/2,0)+tq)+(1-s)tq:0\le t,s \le 1\}$,
\item $T^+:=\{s((1-t)(\sqrt{2}/2,0)+tq)+(1-s)tq:0\le t,s \le 1\}$, 
\item $T:=T^-\cup T^+$,
\item $\Delta^-:=\{s(-\sqrt{2}/2,-t\sqrt{2}/2)+(1-s)(-t\sqrt{2}/2,-t\sqrt{2}/2):0\le t,s \le 1\}$, 
\item $\Delta^0:=\{s(-t\sqrt{2}/2,-t\sqrt{2}/2)+(1-s)(t\sqrt{2}/2,-t\sqrt{2}/2):0\le t,s \le 1\}$, 
\item $\Delta^+:=\{s(\sqrt{2}/2,-t\sqrt{2}/2)+(1-s)(t\sqrt{2}/2,-t\sqrt{2}/2):0\le t,s \le 1\}$,
\item $\Delta:=\Delta^-\cup \Delta^0 \cup \Delta^+$.
\end{itemize}

Now let $P:\mathbb{D}^2\to \mathbb{D}^2$ be a continuous surjection so that:
\\(i) $P$ is the identity on $\partial \mathbb{D}^2\cup \Delta^0$;
\\(ii) $P$ collapses $\Delta^-\cup\Delta^+\cup T^-\cup T^+$ horizontally onto $\Upsilon$ (that is, e.g., $P(s((1-t)(-\sqrt{2}/2,0)+tq)+(1-s)tq)=tq$); and
\\(iii) $P:\mathbb{D}^2\setminus (\Delta^-\cup\Delta^+\cup T^-\cup T^+) \to \mathbb{D}^2\setminus \Upsilon$ is a homeomorphism.
 \\We transfer $P$ to $M$ via $\Psi$: let $\rho:M\to M$ be given by $\rho(x)=x$ for $x\not\in \Psi(\mathbb{D}^2)$ and $\rho(x)=\Psi \circ P\circ \Psi^{-1}(x)$ for $x\in \Psi(\mathbb{D}^2)$.

Let $\alpha:[0,1]\to [0,1]$ be the homeomorphism defined by $f(\Psi(tq))=\Psi((0,\alpha(t))$. For $x\in \Psi(T)$, define $s=s_T(x)$ and 
$t=t_T(x)$ by $x=\Psi(s((1-t)(-\sqrt{2}/2,0)+tq)+(1-s)((1-t)(\sqrt{2}/2,0)+tq))$. For such $x,s,t$ we will write
$x=x_T(s,t)$. If  $x\in \Psi(\Delta)$ define
$s=s_{\Delta}(x)$ and $t=t_{\Delta}(x)$ by $x=\Psi(s(t(-\sqrt{2}/2,-\sqrt{2}/2)+(1-t)(-\sqrt{2}/2,0))+(1-s)(t(\sqrt{2}/2,-\sqrt{2}/2)+(1-t)(\sqrt{2}/2,0)))$. In this case we write $x=x_{\Delta}(s,t)$.

Define $g:M\to M$ by 
\[g(x)=\left\{ \begin{array}{ll}
   f(\rho(x)) & \mbox{if $x\not\in \Psi(T\cup \Delta)$}\\
   x_T(\beta(s),\alpha(t)) & \mbox{if $x\in \Psi(T)$}\\
   x_{\Delta}(\beta(s),\gamma(t)) &  \mbox{ if $x\in \Psi(\Delta).$}
 \end{array} \right.\]
 Then $g$ is a {\em near homeomorphism}, that is, $g$ is a uniform limit of homeomorphisms. It follows
 (see [Br]) that the inverse limit space $\hat{M}:=\{(x_0,x_1,x_2,\ldots):g(x_i)=x_{i-1},i\in \mathbb{N}\}$, with the product topology, is homeomorphic with $M$. Thus the shift $\hat{g}:\hat{M}\to \hat{M}$, defined by $\hat{g}((x_0,x_1,\ldots)):=(g(x_0),x_0,x_1,\ldots)$, is a surface homeomorphism.

Let $B^{\pm}$ be the closed arcs in $M$, $B^-:=\Psi(\{(1-t)(-\sqrt{2}/2,0)+tq:0\le t\le 1\})$ and
$B^+:=\Psi(\{(1-t)(\sqrt{2}/2,0)+tq:0\le t\le 1\})$, with endpoints $p^-:=\Psi((-\sqrt{2}/2,0))$ and $p^+:=
\Psi((\sqrt{2}/2,0))$. Then $g(p^{\pm})=p^{\pm}$ and $g(B^{\pm})\supset B^{\pm}$ and the sets
$R^{\pm}:=\{(x_0,x_1,\ldots):x_n\in B^{\pm}$ for some $n\in \mathbb{N}\}$ are expanding invariant rays in $\hat{M}$, under $\hat{g}$, with fixed endpoints $\bar{p}^{\pm}:=(p^{\pm},p^{\pm},\ldots)$. The open 
topological disk $W:=\{(x_0,x_1,\ldots):x_n\in \Psi(\mathring{T})$ for some $n\in \mathbb{N}\}\subset \hat{M}$ can be seen as a channel between the asymptotic rays $R^-$ and $R^+$. 

Define $\bar{\rho}:\hat{M}\to M$ by: $\bar{\rho}(x_0,x_1,\ldots)=\rho(x_0)$, if $x_0\not\in \Psi(\Delta)$,
and $\bar{\rho}(x_0,x_1,\ldots)=\Psi((s_{\Delta}t_{\Delta}(-\sqrt{2}/2,-\sqrt{2}/2)+(1-s_{\Delta})t_{\Delta}(\sqrt{2}/2,-\sqrt{2}/2))$, if $x_0=x_0(s_{\Delta},t_{\Delta})\in \Psi(\Delta)$. Then $\bar{\rho}$ semi-conjugates $\hat{g}$ with $f$ and $\bar{\rho}$ is one-to-one off  $\bar{\rho}^{-1}(R)=R^-\cup R^+\cup
W\cup C$, $C:=\Psi(\{(t\sqrt{2}/2,0):t\in [-1,1]\})$.

Now suppose that $f:M\to M$ is a surface homeomorphism with expanding attractor $A$. Suppose also
that $p\in A$ is a periodic point, accessible from the complementary domain $U$ of $A$ and suppose there are closed arcs $A^{\pm}$, also periodic under $f$, with $(A^-\cup A^+)\setminus \{p\} \subset U$ and $A^-\cap A^+=\{p\}$. If there is a branch $R$ of the unstable set of $p$ (that is, there is a component $B$ of
$W^u_{\alpha}(p)\setminus \{p\}$ with $f^n(B)\supset B$ for some $n>0$ and  $R=\cup_{k\ge 0}f^{kn}(B)) $ that is inaccessible from $M\setminus A$ and there is $\Psi$ as above with $\Psi(\Delta^0)\setminus{p}\subset U$ then $\hat{A}:=\bar{\rho}^{-1}(A)\setminus (W\cup \mathring{C})$ is an expansive attractor
for $\hat{g}$. Furthermore, $(A,f|_A)$ is derived from pseudo-Anosov (resp., transitive) if and only if $(\hat{A},\hat{g}|_{\hat{A}})$ is. If $p$ was the only fixed point of $f$ accessible from $U$, then $\hat{A}$
has one fewer (than $A$) complementary domains from which a single fixed point is accessible.

The above procedure can be readily modified to construct any derived from pseudo-Anosov attractor
from the unstable foliation of an appropriate pseudo-Anosov homeomorphism.

\bibliography{biblio.bib}

\bibliographystyle{amsalpha}

\end{document}